\documentclass[11pt,letterpaper]{amsart}

\usepackage{mathtools}
\usepackage{xcolor}
\usepackage{amssymb}
\usepackage{comment}
\usepackage[utf8]{inputenc}
\usepackage[cyr]{aeguill}
\usepackage[pdfdisplaydoctitle=true,
            colorlinks=true,
            urlcolor=blue,
            citecolor=blue,
            linkcolor=blue,
            pdfstartview=FitH,
            pdfpagemode= UseNone,
            bookmarksnumbered=true]{hyperref}
\usepackage{todonotes}
\usepackage{enumitem}
\usepackage{mathtools}
\mathtoolsset{showonlyrefs=true}

\usepackage{geometry}
\newtheorem{theorem}{Theorem}
\newtheorem*{theorem*}{Theorem}
\newtheorem{lemma}{Lemma}[section]
\newtheorem{corollary}[lemma]{Corollary}
\newtheorem*{corollary*}{Corollary}
\newtheorem{proposition}[lemma]{Proposition}
\theoremstyle{definition}
\newtheorem{definition}[lemma]{Definition}

\newtheorem{example}[lemma]{Example}

\newtheorem{remark}[lemma]{Remark}

\newcommand{\ad}{\operatorname{ad}}

\DeclareMathOperator{\sgn}{sgn}

\title[$b$-equation]{Local well-posedness of the higher-dimensional $b$-equation}

\author[Justin Valletta]
{Justin Valletta}

\address{Justin Valletta: Florida State University}
\email{jvalletta@fsu.edu}

\date{March 4, 2025.}
\subjclass[2010]{%
}

\begin{document}

\begin{abstract}
The higher-dimensional $b$-equation is a family of PDEs, introduced by Holm and Staley (2003), that describe the motion of shallow water waves in $n$-dimensions. It expresses the invariance of the Lie-transport of the momentum one-form density associated with the fluid in $b$-dimensions. The constant $b$ can also be viewed as a balance parameter between fluid convection and fluid stretching/expansion. In this article, we interpret this family of PDEs as the geodesic equation of a right-invariant affine connection on the diffeomorphism group of $\mathbb{R}^n$. Using this framework and the methods of Ebin and Marsden (1970), we show local well-posedness of the $b$-equation with a Fourier multiplier as the inertia operator. This is achieved by formulating the $b$-equation as a smooth ODE on a Hilbert manifold, applying Picard-Lindel\"{o}f, and transferring back to the smooth category by showing that there is no loss of spatial regularity during the time evolution. 
\end{abstract}

\maketitle

\setcounter{tocdepth}{1}
\tableofcontents

\section{Introduction}
\subsection{Background}
Mechanics on Lie groups began with Poincar\'{e} \cite{poincare1901forme}. In 1966, Arnold \cite{arnold1966} extended this to the infinite-dimensional setting, and derived Euler's equation for the motion of an incompressible, perfect fluid as the geodesic equation of the right-invariant $L^2$-metric on the group of volume-preserving diffeomorphisms. Since then, many other PDEs from physics, mostly hydrodynamics, have been recast as geodesics of a right-invariant metric (and more generally, a right-invariant connection) on an appropriate diffeomorphism group. Such equations are called Euler-Arnold equations. Examples include the Camassa-Holm \cite{camassa1993integrable,kouranbaeva1999camassa,misiolek1998shallow}, the Hunter-Saxton \cite{hunter1991dynamics,lenells2007hunter}, the modified Constantin-Lax-Majda \cite{constantin1985simple,escher2010}, the Korteweg-De Vries \cite{ovsienko1987korteweg}, and the Degasperis-Procesi shallow water equation \cite{Kolev2009}. There are many other interesting examples; see \cite{arnold1998topological,vizman2008} for more.

Interpreting PDEs as geodesic equations on infinite-dimensional Lie groups has been used to obtain local well-posedness results, due originally to Ebin and Marsden \cite{ebin1970groups}, who famously used the geometric interpretation of the incompressible Euler equation to prove local well-posedness on an arbitrary compact manifold. They smoothly extended the metric and the geodesic spray from the Fr\'{e}chet manifold of smooth diffeomorphisms to the Hilbert manifold of $H^k$ diffeomorphisms, allowing one to apply the Picard-Lindel\"{o}f theorem. By showing that there is no loss of spatial regularity during the time evolution, they were able to transfer their local well-posedness result on the finite regularity Sobolev spaces to the smooth category. This idea of smoothly extending the spray was then used by Bauer, Constantin, Escher, Gay-Balmaz, Kolev, and others \cite{constantin2003geodesic,gay2009well,escher2011degasperis,escher2014fractional,bauer2015local,bauer2020wellposedness} to obtain similar local well-posedness results for geodesics on diffeomorphism groups with various inertia operators. 

We are concerned in this article with the $n$-dimensional family of $b$-equations, which take the form 
\begin{equation}\label{b-equation}
 \Omega_t+\nabla_U\, \Omega+(\nabla U)^T\Omega+(b-1)\operatorname{div}(U)\Omega=0,\qquad \Omega=AU,
 \end{equation}
where $U:[0,T)\times\mathbb{R}^n\to\mathbb{R}^n$ is a time-dependent vector field (e.g. fluid velocity), $A$ is a positive-definite, symmetric operator called the \textit{inertia operator}, $b\in\mathbb{R}$ is a dimensionless parameter, and the vector field $\Omega$ is often regarded as the momentum. This equation was first introduced in dimension one by Degasperis, Holm, and Hone \cite{Degasperis2002} with the $H^1$ Sobolev inertia operator $A=1-\partial_x^2$ as a generalization of the Camassa-Holm \cite{camassa1993integrable} and Degasperis-Procesi equations \cite{degasperis1999asymptotic}, shortly after the original derivation of the Degasperis-Procesi equation. It reads
\begin{equation}\label{1D}
m_t+um_x+bu_xm=0,\qquad m=u-u_{xx},
\end{equation}
where $b=2$ corresponds to the Camassa-Holm equation and $b=3$ the Degasperis-Procesi equation, both of which are the only completely integrable members of the one-dimensional $b$-family with the $H^1$ Sobolev inertia operator \cite{Degasperis2002}. Holm and Staley subsequently studied \eqref{1D} in \cite{holm2003wave} where they explored the effect of the parameter $b$ on the stability of various solitary wave solutions and introduced the higher-dimensional version \eqref{b-equation}. The $b$-equation family includes many other well-known PDEs as members, such as the Hunter-Saxton equation \cite{hunter1991dynamics}, the modified Constantin-Lax-Majda equation \cite{constantin1985simple,escher2010}, the $b$-Camassa-Holm equation \cite{degasperis2003integrable}, the Proudman-Johnson equation \cite{okamoto2000some}, and the Okamoto-Sakojo-Wunsch equation \cite{okamoto2008generalization}. The equation one obtains depends on the dimension, the choice of inertia operator, and the value of $b.$

When $b=2$, equation \eqref{b-equation} becomes the EPDiff equation, short for Euler-Poincar\'{e} equation on the diffeomorphism group, introduced by Holm and Marsden \cite{holm2005momentum} in the context of fluid dynamics with the $H^1$ Sobolev inertia operator. As the name implies, the EPDiff equation is the geodesic equation of the right-invariant Riemannian metric induced by the inertia operator $A$ on the diffeomorphism group of $\mathbb{R}^n$. The EPDiff equation has since been studied extensively by Constantin, Escher, Holm, Kolev, Michor, Misiolek, Preston, and others \cite{constantin2002geometric, holm2005momentum,holm2009geodesic,misiolek2010fredholm, escher2014fractional,bauer2015local,kolev2016,bruveris2017completeness,bauer2020fractional,bauer2024liouville}. Notably, Bauer, Escher, and Kolev \cite{bauer2015local} proved local well-posedness of the higher-dimensional EPDiff equation with a fractional order Fourier multiplier as the inertia operator by smoothly extending the spray and using an Ebin and Marsden no-loss no-gain argument. 

When $b\ne2$, however, much less work has been done on the \textit{higher-dimensional} $b$-equation \eqref{b-equation}. Local well-posedness has been established for arbitrary $b$ by Escher and Yin \cite{escher2008well}, but only in dimension one. Moreover, for these values of $b,$ it is known that the $b$-equation in dimension one with the $H^1$ Sobolev inertia operator does not correspond to the geodesic equation of a right-invariant Riemannian \textit{metric} \cite{Kolev2009, escher2010}. One can, however, weaken this geometric interpretation, as Escher and Kolev \cite{Kolev2009,escher2011degasperis} did when they provided a \textit{non-metric} geometric interpretation for the Degasperis-Procesi equation. They showed that this equation is in fact a geodesic equation of a right-invariant affine \textit{connection}. 

In this article, we generalize the result of Escher and Kolev \cite{escher2011degasperis} to the higher-dimensional $b$-equation by showing that it corresponds to the geodesic equation of a right-invariant affine connection, which we can write explicitly. The notion of the geodesic spray therefore makes sense for this equation, and we can demonstrate local well-posedness of the higher-dimensional $b$-equation by smoothly extending the geodesic spray and employing an Ebin and Marsden type of no-loss, no-gain argument.

\subsection{Our contributions} We provide a geometric interpretation of the higher-dimensional $b$-equation \eqref{b-equation} as the geodesic equation of a right-invariant affine connection $\nabla$ on the diffeomorphism group of $\mathbb{R
}^n$; see Theorem \ref{beqnconnection1} in Section \ref{EAsection}. We also highlight the Lie-transport structure of the $b$-equation, and prove a Kelvin-Noether circulation theorem for this equation, akin to Kelvin's circulation theorem for perfect fluids described by the Euler equations; see Section \ref{sec:kelvinnoether}. We then use the geometric framework to establish local well-posedness of the $b$-equation with a Fourier multiplier of order $r\ge1$ as the inertia operator in $H^s(\mathbb{R}^n,\mathbb{R}^n)$ for any $s> n/2+1$ with $s\ge r$. This is achieved by formulating the $b$-equation as an ODE on the Hilbert manifold $T\mathcal{D}^s(\mathbb{R}^n)$, allowing one to obtain local well-posedness in $H^s(\mathbb{R}^n,\mathbb{R}^n)$ via Picard-Lindel\"{o}f. We then show, using the classic Ebin and Marsden no-loss no-gain argument, that there is no loss of spatial regularity during the time evolution of the $b$-equation. This allows one to transfer the local well-posedness result in $H^s(\mathbb{R}^n,\mathbb{R}^n)$ to the smooth category, obtaining local well-posedness in $H^\infty(\mathbb{R}^n,\mathbb{R}^n).$ Our main result is the following theorem.

\begin{theorem*}[Corollary~\ref{mainresult} in Section \ref{wellposednesssection}]
    The n-dimensional $b$-equation \eqref{b-equation} with a Fourier multiplier of class $\mathcal{E}^r$ with $r\geq1$ as the inertia operator has for any initial data $u_0\in H^\infty(\mathbb{R}^n,\mathbb{R}^n)$, a unique non-extendable smooth solution $u\in C^\infty(J,H^\infty(\mathbb{R}^n,\mathbb{R}^n))$. 
\end{theorem*}

\subsection{Acknowledgments} I would like to graciously thank Martin Bauer and Stephen Preston for dedicating significant time and patience in helping me understand many of the beautiful concepts herein. I would also like to give a special thanks to Darryl Holm, who helped me to understand and formulate the Lie-transport formula and the Kelvin-Noether circulation theorem. My correspondence with him dramatically increased my appreciation for the physical meaning of this transport equation, and Section \ref{sec:kelvinnoether} would not have been possible without his assistance. This work was partially supported by the BSF under Grant No. 2022076.  

\section{Background material}\label{background}

Here we collect the necessary background material for the results in this article. 
This consists primarily of defining the appropriate Sobolev spaces and the corresponding groups of diffeomorphisms on which our equations will be formulated. We describe how to extend the usual definition of integer-order Sobolev spaces in terms of weak (or distributional) derivatives to fractional orders. We also briefly comment on the manifold structure of the relevant diffeomorphism groups, along with a few notable properties that will be used subsequently. Lastly, we define the class of inertia operators for which our local well-posedness result for the $b$-equation will hold.

\subsection{The Sobolev space $H^s(\mathbb{R}^n,\mathbb{R}^n)$}

We use diffeomorphism groups that are defined in terms of Sobolev spaces, which we will now discuss. The space $H^s(\mathbb{R}^n,\mathbb{R}^n)$ is the Hilbert space 
\[H^s(\mathbb{R}^n,\mathbb{R}^n):=\{f=(f_1,\dots,f_n) \, | \,f_i\in H^s(\mathbb{R}^n,\mathbb{R}),i=1,\dots,n\}\]
with the $H^s$-norm given by
\[\|f\|_{H^s}:=\left(\sum_{i=1}^n\|f_i\|_{H^s}\right)^{1/2},\]
where $\|f_i\|_{H^s}$ is the norm on $H^s(\mathbb{R}^n,\mathbb{R}).$ When $s\in\mathbb{Z}_{\geq0},$ then $H^s(\mathbb{R}^n,\mathbb{R})$ is the Hilbert space of functions $g\in L^2(\mathbb{R}^n,\mathbb{R})$ such that the distributional derivatives $D^\alpha g$ live in $L^2(\mathbb{R}^n,\mathbb{R})$ for every multi-index $\alpha\in\mathbb{Z}^n_{\geq0}$ of order $|\alpha|\leq s.$ It has norm
\[\|g\|_s:=\left(\sum_{|\alpha|\leq s}\|D^\alpha g\|_{L^2}^2\right)^{1/2}=\left(\sum_{|\alpha|\leq s}\int_{\mathbb{R}^n}|D^\alpha g|^2\,dx\right)^{1/2}.\]
An alternative characterization of the space $H^s(\mathbb{R}^n,\mathbb{R})$ in terms of the Fourier transform is useful if one wants to allow $s\in\mathbb{R}_{\geq0}.$ For $f\in L^2(\mathbb{R}^n,\mathbb{R})$, define the Fourier transform $\mathcal{F}f$ by
\[(\mathcal{F}f)(\xi):=(2\pi)^{-n/2}\int_{\mathbb{R}^n}f(x)e^{-i x\cdot\xi}\,dx,\]
which then has inverse given by 
\[(\mathcal{F}^{-1}g)(x):=(2\pi)^{-n/2}\int_{\mathbb{R}^n}g(\xi)e^{i x\cdot\xi}\,d\xi.\]
Whenever the integral exists, it holds that $\mathcal{F}f\in L^2(\mathbb{R}^n,\mathbb{R})$ by the Plancherel theorem $\|f\|_{L^2}=\|\mathcal{F}f\|_{L^2}$. Arguably the most notable property of the Fourier transform is that the differential operator $D^\alpha$ becomes the multiplication operator $\mathcal{F}f\mapsto (i\xi)^\alpha\mathcal{F}f$ under the transform. Concretely, $\mathcal{F}(D^\alpha f)=(i\xi)^\alpha\mathcal{F}f$. These properties imply that when $f\in H^s(\mathbb{R}^n,\mathbb{R}),$ we have $\xi^\alpha\mathcal{F}f\in L^2(\mathbb{R}^n,\mathbb{R})$. In particular, $\xi^{s}_j\mathcal{F}f\in L^2(\mathbb{R}^n,\mathbb{R})$ for every $j\in\{1,\dots,n\}.$ Now apply Holder's inequality with exponents $p$ and $s$ to the vectors $u:=(1,\xi^2_1,\dots,\xi^2_n),v:=(1,\dots,1)\in\mathbb{R}^{n+1}$ to find
\[1+|\xi|^2=u\cdot v\leq \|u\|_p\|v\|_s=(n+1)^{1/p}(1+\xi^{2s}_1+\dots+\xi^{2s}_n)^{1/s}.\]
This implies
\[\int_{\mathbb{R}^n}(1+|\xi|^2)^s|\mathcal{F}f|^2\,dx\leq(n+1)^{s/p}\int_{\mathbb{R}^n}(1+\xi^{2s}_1+\dots+\xi^{2s}_n)|\mathcal{F}f|^2\,dx<\infty,\]
from which we see that $(1+|\xi|^2)^{s/2}\mathcal{F}f\in L^2(\mathbb{R}^n,\mathbb{R})$ if and only if $f\in H^s(\mathbb{R}^n,\mathbb{R})$. Moreover, if we define the norm $\|f\|^*_{H^s}$ by 
\[\|f\|^*_{H^s}:=\|(1+|\xi|^2)^{s/2}\mathcal{F}f\|_{L^2},\]
one can show it is equivalent to the $\|f\|_{H^s}$ norm in that
\[C_s^{-1}\|f\|_{H^s}\le\|f\|^*_{H^s}\le C_s\|f\|_{H^s},\]
for some constant $C_s\ge1$. This allows us to define $H^s(\mathbb{R}^n,\mathbb{R})$ for $s\in\mathbb{R}_{\geq0}$ by
\[H^s(\mathbb{R}^n,\mathbb{R})=\{f\in L^2(\mathbb{R}^n,\mathbb{R})\,|\,\|f\|^*_{H^s}<\infty\}.\]
These non-integer order Sobolev spaces satisfy a few notable properties which we collect in the following proposition. 

\begin{proposition}[Inci, Kappeler, Topalov \cite{inci2012regularity}; Section 2]\label{sobolevproperties}
    For any $s\in\mathbb{R}_{\ge0}$, the following hold.
    \begin{enumerate}
        \item The space of smooth functions with compact support $C_c^\infty(\mathbb{R}^n,\mathbb{R}^n)$ is dense in $H^s(\mathbb{R}^n,\mathbb{R}^n).$
        \item If $s\ge n/2$, then for any integer $r$, the space $H^{s+r}(\mathbb{R}^n,\mathbb{R}^n)$ can be embedded onto the space $C_0^r(\mathbb{R}^n,\mathbb{R}^n)$ of $C^r$ functions vanishing at infinity.
        \item If $s>n/2$ and $q\ge 1$ with $q\le s$, then pointwise multiplication extends to a bounded bilinear mapping 
        \[H^s(\mathbb{R}^n,\mathbb{R}^n)\times H^q(\mathbb{R}^n,\mathbb{R}^n)\to H^q(\mathbb{R}^n,\mathbb{R}^n).\]
    \end{enumerate}
\end{proposition}

The first property in the above proposition shows that the spaces $H^s(\mathbb{R}^n,\mathbb{R}^n)$ can also be defined as the closure of $C_c^\infty(\mathbb{R}^n,\mathbb{R}^n)$ with respect to the $\|\cdot\|^*_{H^s}$ norm.

\subsection{The group of smooth Sobolev diffeomorphisms}
The methods used herein fundamentally rely on the diffeomorphism group being a regular Fr\'{e}chet-Lie group. A Fr\'{e}chet-Lie group is an infinite dimensional Lie group modeled on a Fr\'{e}chet space. Such a Lie group is called \textit{regular} if every smooth path $v:[0,1]\to\mathfrak{g}$ has a Lagrangian flow that can be integrated, i.e. every $v\in\mathfrak{g}$ generates a flow that can be uniquely determined. The groups of diffeomorphisms of smooth compact manifolds are regular, but this is not the case for the full diffeomorphism group $\operatorname{Diff}^\infty(\mathbb{R}^n)$ of smooth, orientation-preserving diffeomorphisms on $\mathbb{R}^n$. Hence we will work on a subgroup of $\operatorname{Diff}^\infty(\mathbb{R}^n)$ that has this regularity property. There are a few different options, but we will use the subgroup of smooth, orientation-preserving diffeomorphisms $\operatorname{Diff}_{H^\infty}(\mathbb{R}^n)$ which differ from the identity by a function in the Sobolev space 
\[H^\infty(\mathbb{R}^n,\mathbb{R}^n)
:=\bigcap_{s\ge0}H^s(\mathbb{R}^n,\mathbb{R}^n).\]
Explicitly, we define
\[\operatorname{Diff}_{H^\infty}(\mathbb{R}^n):=\{\operatorname{id}+f\,:\,f\in H^\infty(\mathbb{R}^n,\mathbb{R}^n)\text{ and } \operatorname{det}(\operatorname{id}+df)>0\}.\]
This space was proposed by Djebail and Hermas \cite{hermas2010existence}, and therein it was shown that it is a regular Fr\'{e}chet-Lie group. Its Lie algebra is the space $\mathfrak{X}_\infty(\mathbb{R}^n)=H^\infty(\mathbb{R}^n,\mathbb{R}^n)$ of $H^\infty$ vector fields \cite{bauer2015local}. We must also introduce the Hilbert approximations $\mathcal{D}^s(\mathbb{R}^n)$ of 
$\operatorname{Diff}_{H^\infty}(\mathbb{R}^n)$, defined for $s>n/2+1$ by 
\[
\mathcal{D}^s(\mathbb{R}^n):=
\{\operatorname{id}+f\,:\,f\in H^s(\mathbb{R}^n,\mathbb{R}^n)\text{ and } \operatorname{det}(\operatorname{id}+df)>0\}.\]
We need the assumption $s>n/2+1$ so that one can define a smooth chart from $\mathcal{D}^s(\mathbb{R}^n)$ to $H^s(\mathbb{R}^n,\mathbb{R}^n)$ via $\varphi\mapsto\varphi-\operatorname{id}$ \cite{inci2012regularity}. Hence, $\mathcal{D}^s(\mathbb{R}^n)$ has the structure of a smooth Hilbert manifold, but not a Lie group since composition and inversion are not smooth \cite{bauer2015local}. $\operatorname{Diff}_{H^\infty}(\mathbb{R}^n)$ therefore has the advantage of being a Lie group, but only a Fr\'{e}chet one, while $\mathcal{D}^s(\mathbb{R}^n)$ has the advantage of being a Hilbert manifold, but not a Lie group. Note also that the tangent bundle $T\mathcal{D}^s(\mathbb{R}^n)$ is trivial:
    \[T\mathcal{D}^s(\mathbb{R}^n)\cong \mathcal{D}^s(\mathbb{R}^n)\times H^s(\mathbb{R}^n,\mathbb{R}^n)\]
    since $\mathcal{D}^s(\mathbb{R}^n)$ is an open subset of the Hilbert space $H^s(\mathbb{R}^n,\mathbb{R}^n)$ \cite{bauer2015local}.

\subsection{Fourier multipliers}\label{fouriermultipliersection}

We follow the presentation in \cite{bauer2015local}. First observe that if $A$ is a constant coefficient differential operator then
\begin{equation}\label{multiplier}
\widehat{Au}(\xi)=a(\xi)\hat{u}(\xi)
\end{equation}
for some polynomial $a:\mathbb{R}^n\to\mathbb{C}$. Notice that we can view the image of this polynomial $a(\xi)$ as a continuous linear transformation that acts on $\hat{u}\in\mathbb{C}^n$ via multiplication. Motivated by this observation, we define a \textit{Fourier multiplier} $A$ by the formula \eqref{multiplier} for an arbitrary complex-valued function  $a(\xi)\in \mathcal{L}(\mathbb{C}^n,\mathbb{C}^n)$, that we will view as a continuous linear transformation. If $a$ and $u$ are sufficiently regular, we can invert the Fourier transform to recover $A$ explicitly. The function $a$ is called the \textit{symbol} of the Fourier multiplier $A$, and it clearly determines the operator $A$. Hence Fourier multipliers are often denoted by $A=\textbf{op}(a(\xi)).$ 

\begin{example}
    The Hilbert transform $H$ of a real-valued function $u(t)$ is typically defined by
    \[H(u)(t):=\frac{1}{\pi}\,\text{p.v.}\int_{-\infty}^\infty\frac{u(\tau)}{t-\tau}\,d\tau.\]
    Notice that $H(u)(t)=u(t)*\frac{1}{\pi t}$. By the convolution theorem for the Fourier transform, one finds
    \[\widehat{H(u)}(\xi)=-i \sgn(\xi)\hat{u}(\xi),\]
    and hence the Hilbert transform is a multiplier with symbol $a(\xi)=-i \sgn(\xi).$
\end{example}

\begin{example}
The inertia operator that gives rise to the $H^s$ Sobolev metric is defined by
\[\Lambda^{2s}:=\operatorname{\bf{op}}((1+|\xi|^2)^s),\]
and has order at least 1 if we require $s\geq1/2.$ 
\end{example}

We will restrict ourselves to a special class of multipliers which are well-defined on $H^\infty(\mathbb{R}^n,\mathbb{R}^n)$. They are general enough to include the multipliers we care about (e.g. Sobolev inertia operator).

\begin{definition}
    Fix $r\in\mathbb{R}.$ A Fourier multiplier $A=\textbf{op}(a(\xi))$ is of class $\mathcal{S}^r$ if $a\in C^\infty(\mathbb{R}^n,\mathcal{L}(\mathbb{C}^n,\mathbb{C}^n))$ satisfies for each multi-index $\alpha\in \mathbb{N}^n$,
    \[\|\partial^\alpha a(\xi)\|\lesssim(1+|\xi|^2)^{(r-|\alpha|)/2},\]
    where $\partial^\alpha a(\xi)$ is the $\alpha^{th}$ Fr\'{e}chet derivative of the linear transformation $a(\xi)$, and $\| . \|$ the operator norm.
\end{definition}

\begin{example} 
Any linear, constant-coefficient differential operator of order $r$ is in the class $\mathcal{S}^r$. The Sobolev inertia operator $A=\textbf{op}(1+|\xi|^2)^{r/2}$ is also in this class.
\end{example}

\begin{definition}
    A Fourier multiplier $A=\textbf{op}(a(\xi))$ in the class $\mathcal{S}^r$ is called \textit{elliptic} if 
    $a(\xi)\in\mathcal{GL}(\mathbb{C}^n,\mathbb{C}^n)$ for all $\xi\in\mathbb{R}^n$ and 
    \[\|a(\xi)^{-1}\|\lesssim (1+|\xi|^2)^{-r/2},\quad \forall\xi\in\mathbb{R}^n.\]
\end{definition}

\begin{definition}
    A Fourier multiplier $A\in\mathcal{L}(H^\infty(\mathbb{R}^n,\mathbb{R}^n),H^\infty(\mathbb{R}^n,\mathbb{R}^n))$ is in the class $\mathcal{E}^r$ if the following hold:
    \begin{enumerate}
        \item $A=\textbf{op}(a(\xi))$ is of class $\mathcal{S}^r$;
        \item $A=\textbf{op}(a(\xi))$ is elliptic;
        \item The symbol $a(\xi)$ is Hermitian and positive-definite for all $\xi\in\mathbb{R}^n.$
    \end{enumerate}
\end{definition}

\begin{remark}
    The class $\mathcal{S}^r$ with $r\ge1$ is required to show boundedness of the $d$-th Fr\'{e}chet differentials of the twisted operator $A_{\varphi}=R_{\varphi}\circ A\circ R_{\varphi^{-1}}$. This boundedness is used, as in \cite{bauer2015local}, to show that the map 
\[\varphi\mapsto A_\varphi:=R_{\varphi}\circ A\circ R_{\varphi^{-1}},\qquad \mathcal{D}^s(\mathbb{R}^n)\to\mathcal{L}(H^s(\mathbb{R}^n,\mathbb{R}^n),H^{s-r}(\mathbb{R}^n,\mathbb{R}^n))\]
is smooth when $s>n/2+1$ and $s-r\ge0$. We will use this fact to show smoothness of the geodesic spray for the higher-dimensional $b$-equation. Many of the conditions defining the class $\mathcal{E}^r$ originate from the metric case (that is, when the PDE of interest is a geodesic with respect to the Levi-Civita connection). In that context, we must require certain conditions on the inertia operator to ensure that the induced metric is smooth. For example, the requirement that the symbol $a(\xi)$ be Hermitian and positive-definite ensures that an inner product of the form 
    \[\langle u_1,u_2 \rangle:=\int_{\mathbb{R}^n}Au_1\cdot u_2\,dx,\qquad u_1,u_2\in H^\infty(\mathbb{R}^n,\mathbb{R}^n)\]
    is $L^2$-symmetric and positive definite \cite{bauer2015local}. In the non-metric case (when the connection is not Levi-Civita), it seems we can relax this requirement. The invertibility part of the ellipticity condition is necessary to demonstrate smoothness of the spray (see Proposition \ref{system1}), but we may be able to omit the condition on the norm of $a(\xi)^{-1}$. For more information on these classes of Fourier multipliers, see \cite{escher2014fractional,bauer2015local}. 
\end{remark}

\section{Geometric interpretations of the  \texorpdfstring{$b$}{b}-equation and a Kelvin-Noether circulation theorem} \label{interpretations}

In this section, we begin by describing the notion of a non-metric Euler-Arnold equation, first introduced by Escher and Kolev \cite{escher2011degasperis}, and provide a geometric interpretation of the higher-dimensional $b$-equation as a non-metric Euler-Arnold equation on the diffeomorphism group $\operatorname{Diff}_{H^\infty}(\mathbb{R}^n)$; see Theorem \ref{beqnconnection1}. We then describe the Lie-transport structure of the $b$-equation, and provide a Kelvin-Noether circulation theorem for this transport equation; see Theorem \ref{kelvinnoether}.

\subsection{Non-metric Euler-Arnold equation}
\label{EAsection} 

There is a well-known equivalence between the geodesic equation on a Lie group and a first-order equation on the Lie algebra called the \textit{Euler-Arnold equation}, first introduced by Arnold in \cite{arnold1966} (see also \cite{vizman2008, misiolek2010fredholm}). 
For any smooth Riemannian manifold $(G,g)$, the energy functional of a smooth curve $c:[0,1]\to G$ is defined by 
\begin{equation}\label{energy}
E(c):=\frac{1}{2}\int_0^1g_{c(t)}(\dot{c}(t),\dot{c}(t))\,dt.
\end{equation}
A \textit{metric} geodesic is a critical point of this energy functional. If $G$ is a regular Fr\'{e}chet-Lie group, we can define a (weak) right-invariant metric on $G$ by specifying an inner product $\langle \,,\,\rangle$ on its Lie algebra $\mathfrak{g}$ and extending to all tangent spaces by right translations. That is, for $\xi,\eta\in T_xG$ we define $g_x(\xi,\eta):=\langle\xi x^{-1},\eta x^{-1} \rangle$. This inner product on the Lie algebra is typically prescribed in terms a symmetric linear operator $A:\mathfrak{g}\to\mathfrak{g}^*$ called the inertia operator, which defines an inner product $\langle\,,\, \rangle$ on $\mathfrak{g}$ via 
\begin{equation}\label{rightinvariantmetric}
\langle\xi,\eta \rangle:=(A\xi,\eta)=(A\eta,\xi)\qquad \forall\xi,\eta\in\mathfrak{g},
\end{equation}
where $(\,,\,)$ denotes the natural pairing of the Lie algebra $\mathfrak{g}$ and its dual $\mathfrak{g}^*.$ 
The geodesic equation on the Lie group $G$ is equivalent to an equation on the Lie algebra $\mathfrak{g}$, called the Euler-Arnold equation, in the following sense. 

\begin{theorem}[Euler-Arnold equation \cite{arnold1966}] \label{metricEA}
    Let $G$ be a regular Fr\'{e}chet-Lie group equipped with a (weak) right-invariant metric $g$. A curve $c:[0,1]\to G$ is a geodesic with respect to the metric $g$ if and only if its right logarithmic derivative $u(t)=\dot{c}(t)c(t)^{-1}:[0,1]\to\mathfrak{g}$ satisfies the Euler-Arnold equation 
    \[\dot{u}(t)=-\operatorname{ad}_u^\top u,\]
    where $\operatorname{ad}_u^\top$ is the formal adjoint of the operator $\operatorname{ad}_u:v\to[u,v]$ with respect to the inner product $\langle\,,\,\rangle$ on the Lie algebra $\mathfrak{g}$, i.e.
    \[\langle\operatorname{ad}_u^\top v,w\,\rangle=\langle v,\operatorname{ad}_u w\,\rangle.\]
\end{theorem} 

\begin{proof}
    Let $c:(-\delta,\delta)\times[0,1]\to G$ be a variation of the curve $c$ with fixed endpoints and variational parameter $\varepsilon\in(-\delta,\delta)$. Using right-invariance of the metric $g$, the energy of this variational family of curves is given by
    \[E(c)=\frac{1}{2}\int_0^1 \left
    \langle\dot{c}(\varepsilon,t)c(\varepsilon,t)^{-1},\dot{c}(\varepsilon,t)c(\varepsilon,t)^{-1}\right\rangle\,dt.\]
    Taking the first variation of the energy functional gives
    \[\delta E(c)=\int_0^1 \left\langle u(t),\frac{d}{d\varepsilon}\Big|_{\varepsilon=0}\dot{c}(\varepsilon,t)c(\varepsilon,t)^{-1}\right\rangle\,dt=\int_0^1 \left\langle u(t),\dot{v}(t)-\operatorname{ad}_{u(t)}v(t)\right\rangle\,dt,\]
    where $v(t)=(\partial_\varepsilon c(\varepsilon,t)|_{\varepsilon=0})c(t)^{-1}.$ Now integrate by parts and use the definition of the adjoint transpose to deduce
    \[\delta E(c)=-\int_0^1 \left\langle \dot{u}(t)+\operatorname{ad}^\top_{u(t)}u(t),v(t)\right\rangle\,dt.\]
    As geodesics on $G$ correspond to vanishing variations of $E(c)$, the velocity field $u(t)$ must satisfy
    \[\dot{u}(t)=-\operatorname{ad}^\top_{u(t)}u(t).\]
\end{proof}

\begin{remark}
The right-invariant metric defined by \eqref{rightinvariantmetric} is only a \textit{weak} Riemannian metric. It does not in general induce an isomorphism between a given tangent space $T_gG$ and its dual (as is the case if $G$ is of finite dimension).  Moreover, a covariant derivative that is torsion-free and compatible with the metric may not exist \cite{constantin2002geometric}. 
\end{remark}

Euler-Arnold equations were originally derived with respect to the Levi-Civita connection of a one-sided invariant \textit{Riemannian metric} on a Lie group, but, as pointed out by Escher and Kolev \cite{escher2011degasperis}, the theory also holds in the more general setting of a one-sided invariant \textit{affine connection}, which we now describe. We can define a geodesic $\gamma$ in terms of only a connection $\nabla$, by
\[\nabla_{\dot{\gamma}}\dot{\gamma}=0,\]
without reference to a metric. In this case, one has a result analogous to Theorem \ref{metricEA}, which we present below. First we must introduce some terminology. If $\varphi$ is a diffeomorphism on $G$, we say the affine connection $\nabla$ is \textit{invariant} under $\varphi$ if $\varphi^*(\nabla_XY)=\nabla_{\varphi^*X}\varphi^*Y$, where $\varphi^*X$ is the vector field defined pointwise by
\[(\varphi^*X)(x):=T_{\varphi(x)}\varphi^{-1}X_{\varphi(x)}.\]
An affine connection on $G$ is called \textit{right-invariant} if it is invariant under right translations $h\mapsto hg$ in $G$, and similarly for left translations. There is a canonical connection on any Lie group which is bi-invariant, defined by
\[\nabla^0_{\xi_u}\xi_v=\frac{1}{2}[\xi_u,\xi_v],\]
where $\xi_w$ is the right-invariant vector field on $G$ generated by the vector $w\in\mathfrak{g}$. 
To every right-invariant affine connection $\nabla$ on $G$, we can associate the bilinear form
\[B(X,Y)=\nabla_XY-\nabla^0_XY.\]
Conversely, each bilinear form $B:\mathfrak{g}\times\mathfrak{g}\to\mathfrak{g}$ induces a unique right-invariant affine connection $\nabla$ on $G$ which is given on the Lie algebra $\mathfrak{g}$ by
\begin{equation}\label{rightinvariantconnection}
\nabla_{u}v=\frac{1}{2}[u,v]+B(u,v).
\end{equation}
This extends to a connection on the whole tangent bundle 
\begin{equation}\label{eqn:connection}
\nabla_{\xi_u}\xi_v:=
\xi_{\nabla_{u}v}.
\end{equation}

\begin{example}
Recall that the group of Sobolev diffeomorphisms $\operatorname{Diff}_{H^\infty}(\mathbb{R}^n)$ has as its Lie algebra $H^\infty(\mathbb{R}^n,\mathbb{R}^n).$ For any smooth diffeomorphism $\varphi\in \operatorname{Diff}_{H^\infty}(\mathbb{R}^n)$, the tangent vector $u\in H^\infty(\mathbb{R}^n,\mathbb{R}^n)$ induces a right-invariant vector field, defined pointwise by $(\xi_u)_\varphi:= d(R_\varphi)_{\operatorname{id}}u.$ If $\psi$ is a curve in $\operatorname{Diff}_{H^\infty}(\mathbb{R}^n)$ such that $\psi(0)=\operatorname{id}$ and $\partial_t\psi(0)=u,$ then 
\[
d(R_\varphi)_{\operatorname{id}} u=\frac{d}{dt}\Big|_{t=0} R_\varphi(\psi(t,x))=\frac{d}{dt}\Big|_{t=0} \psi(t,\varphi(x))=\partial_t\psi(0,\varphi(x))=u\circ\varphi,
\]
so that $(\xi_u)_\varphi=u\circ\varphi$. Then for any $v_1,v_2\in T_\varphi\operatorname{Diff}_{H^\infty}(\mathbb{R}^n)$ we have by \eqref{eqn:connection} that
\begin{equation}\label{eqn:connectiondiffeo}
\nabla_{v_1}v_2:=R_\varphi\left(\frac{1}{2}[v_1\circ\varphi^{-1},v_2\circ\varphi^{-1}]+B(v_1\circ\varphi^{-1},v_2\circ\varphi^{-1})\right)
\end{equation}
since $v_1\circ\varphi^{-1},v_2\circ\varphi^{-1}\in H^\infty(\mathbb{R}^n,\mathbb{R}^n)$ generate the right-invariant vector fields $v_1,v_2$ (allowing $\varphi$ to vary).
\end{example}

The velocity field $u(t)$ of geodesics with respect to a right-invariant connection also satisfy an Euler-Arnold equation, analogous to the metric case.

\begin{theorem}[Non-metric Euler-Arnold equation; Escher and Kolev \cite{escher2011degasperis}]\label{connectionEA}
    A smooth curve $c:[0,1]\to G$ is a geodesic for a right-invariant linear connection $\nabla$ defined by \eqref{rightinvariantconnection} if and only if its right logarithmic derivative $u(t)=\dot{c}(t)c(t)^{-1}:[0,1]\to\mathfrak{g}$ satisfies the Euler-Arnold equation
    \begin{equation}
        \partial_tu=-B(u,u).
    \end{equation}
    In particular, every bilinear operator $B$ on the Lie algebra $\mathfrak{g}$ corresponds to the geodesic flow of a right-invariant symmetric linear connection on $G$.
\end{theorem}

\begin{remark}
The Arnold bilinear operator $B$ is defined by
\begin{equation}\label{arnoldbilinear}
B(u,v)=\frac{1}{2}\left(\ad_u^\top v+\ad_v^\top u \right).
\end{equation}
With this bilinear form, the Euler-Arnold equation in Theorem \ref{metricEA} can be written in the same form as in Theorem \ref{connectionEA}. Notice that \eqref{arnoldbilinear} is defined in terms of the metric on $G$. Arnold showed that a necessary and sufficient condition for the existence of a covariant derivative compatible with a right-invariant metric on $\operatorname{Diff}_{H^\infty}(\mathbb{R}^n)$ is the existence of the Arnold bilinear form \eqref{arnoldbilinear}. Since the metric on $\operatorname{Diff}_{H^\infty}(\mathbb{R}^n)$ is only a weak Riemannian metric, the existence of such a covariant derivative is not guaranteed a priori. The bilinear form $B$ in the Euler-Arnold equation is given by \eqref{arnoldbilinear} if and only if the corresponding geodesics are critical points of the energy functional \eqref{energy} with respect to a right-invariant metric.
\end{remark}

We can now observe that the higher-dimensional $b$-equation is the Euler-Arnold equation of a linear connection which is given by the construction above.

\begin{theorem}\label{beqnconnection1}
If the inertia operator is invertible, then the b-equation \eqref{eqn:bequationdimn} is the Euler-Arnold equation of the right-invariant symmetric affine connection $\nabla$ on $\operatorname{Diff}_{H^\infty
}(\mathbb{R}^n)$ defined on the Lie algebra $H^\infty(\mathbb{R}^n,\mathbb{R}^n)$ by
\eqref{rightinvariantconnection}
where the bilinear form $B$ reads
\begin{equation*}
B(u,v)=\frac{1}{2}A^{-1}\left(\nabla_u(Av)+\nabla_v(Au)+(\nabla u)^T(Av)+(\nabla v)^T(Au)+(b-1)\left(\operatorname{div}(u)Av+\operatorname{div}(v)Au\right)\right)
\end{equation*}
This connection extends to the entire tangent bundle $T\operatorname{Diff}_{H^\infty
}(\mathbb{R}^n)$ in the usual way, cf. \eqref{eqn:connection}.
\end{theorem}

\begin{proof}
Let $A$ be an invertible, continuous linear operator on $H^\infty(\mathbb{R}^n,\mathbb{R}^n)$. The $b$-equation \eqref{b-equation} can then be written as
\begin{equation}\label{eqn:bequationdimn}
u_t=-A^{-1}\left(\nabla_u(Au)+(\nabla u)^T(Au)+(b-1)\operatorname{div}(u)Au\right),
\end{equation}
where $u:[0,T)\times\mathbb{R}^n\to\mathbb{R}^n$ is a time-dependent vector field. Define the bilinear form $B:H^\infty(\mathbb{R}^n,\mathbb{R}^n)\times H^\infty(\mathbb{R}^n,\mathbb{R}^n)\to H^\infty(\mathbb{R}^n,\mathbb{R}^n)$ as in the statement of the theorem. Then $B$ induces a unique right-invariant symmetric connection $\nabla$ given by \eqref{rightinvariantconnection}. Observe now that \eqref{eqn:bequationdimn} is of the form
\[u_t=-B(u,u).\]
By Theorem \ref{connectionEA}, it corresponds to the Euler-Arnold equation of the right-invariant symmetric connection $\nabla$ on $\operatorname{Diff}_{H^\infty
}(\mathbb{R}^n)$.
\end{proof}

\begin{remark}
When $b=2$, the connection $\nabla$ defined in Theorem \ref{beqnconnection1} is the Levi-Civita connection; see e.g. \cite{bauer2015local} where EPDiff is derived as an extremum of the energy functional \eqref{energy} with respect to the right-invariant metric induced by a Fourier multiplier $A$. Moreover, in this case, $B$ is Arnold's bilinear form. For $b\neq2$, we have just shown in Theorem \ref{beqnconnection1} that the $b$-equation is a non-metric Euler-Arnold equation, and the connection is not Levi-Civita.
\end{remark}

\subsection{Lie-transport and a Kelvin-Noether circulation theorem}\label{sec:kelvinnoether}
The higher-dimensional $b$-equation describes the motion of shallow water waves in $n$-dimensions, where the constant $b$ can be viewed as a balance parameter between fluid convection, captured by the covariant derivative term, and fluid stretching/expansion, captured by the remaining quadratic terms \cite{holm2003wave}. The parameter $b$ has another important interpretation, as described by Holm and Staley \cite{holm2003wave}, which we now recall. In dimension one, if the function $m^{1/b}$ is well-defined, \eqref{1D} may be written as the conservation law
\[\partial_t(m^{1/b})+\partial_x(m^{1/b}u)=0,\]
implying conservation of the quantity
\[\int_{-\infty}^\infty m^{1/b}(x,t)\,dx=\int_{-\infty}^\infty m^{1/b}(x,0)\,dx.\]
For a Lagrangian flow $x(X,t)$ satisfying $dx=(m^{1/b})^{-1}\,dX+u\,dt$, this implies
\begin{equation}\label{conserved}
m^{1/b}(x,t)dx=m^{1/b}(X,0)dX
\end{equation}
since $x(X,0)=X.$ The tensor product of each side of this equation $b$ times then yields
\[m(x,t)dx^{\otimes b}=m(X,0)dX^{\otimes b},\] which one can differentiate in time at a constant Lagrangian coordinate $X$ to find
\begin{equation}\label{lietransport}
\frac{d}{dt}\Big|_{X}\left(m(x,t)dx^{\otimes b}\right)=(m_t+um_x+bu_xm)\,dx^{\otimes b}=0.
\end{equation}
Here, to differentiate the coordinate one-forms $dx$, we have used the fact that $dx/dt|_{X}=u(x,t)$, which follows at once from the defining relation $dx=(m^{1/b})^{-1}\,dX+u\,dt$. The $b$-equation \eqref{1D} (with any inertia operator) therefore expresses the invariance of the Lie-transport of the one-form density $m(x,t)dx$ in $b$-dimensions \cite{holm2003wave}. When $b$ is a negative integer, \eqref{1D} can be interpreted as
\[\frac{d}{dx}\Big|_{X}\left(m(x,t)(\partial_x)^{\otimes -b}\right)=0.\]
If $b=-1$, for example, this gives
\[\frac{d}{dx}\Big|_{X}(m\partial_x)=(m_t+um_x-u_xm)\partial_x=0,\]
where $(um_x-u_xm)\partial_x=[u\partial_x,m\partial_x].$ This interpretation can be lifted to the higher dimensional $b$-equation by regarding the momentum $\Omega:\mathbb{R}^n\times [0,T)\to\mathbb{R}^n$ as the vector coefficient of an invariant one-form density $\Omega_i(x,t)dx^i$ defined by
\[\Omega(x,t)\cdot dx\otimes (dV)^{\otimes (b-1)}=\Omega(X,0)\cdot dX\otimes (d\tilde{V})^{\otimes (b-1)},\]
where $dV=dx^1\wedge\dots\wedge dx^n$ and $d\tilde{V}=dX^1\wedge\dots\wedge dX^n$ are the Eulerian and Lagrangian volume forms, respectively \cite{holm2003wave}. Differentiating this equation in time at a constant Lagrangian coordinate $X$ yields 
\begin{equation}\label{lietransport1}
\begin{split}
    \frac{d}{dt}\Big|_{X}\left(\Omega(x,t)\cdot dx\otimes (dV)^{\otimes (b-1)}\right)=&\frac{d\Omega}{dt}\Big|_{X}\cdot dx\otimes (dV)^{\otimes (b-1)}\\
    &+\Omega(x,t)\cdot du\otimes (dV)^{\otimes (b-1)}\\
    &\qquad+\Omega(x,t)\cdot dx\otimes \frac{d}{dt}\Big|_{X}\left((dV)^{\otimes (b-1)}\right)=0,
    \end{split}
    \end{equation} 
where we have again used $dx/dt|_{X}=u$ to differentiate the line element $dx=\langle dx^1,\dots,dx^n \rangle$, noting that $du=(\partial_iu)dx^i.$ Using these same formulae, one readily finds the Lagrangian time derivative of the volume form to be
\begin{align*}
\frac{d}{dt}\Big|_{X}(dV)=(\operatorname{div}u)dV.
\end{align*}
Lastly, since $x=x(X,t)$ with $X$ constant, we have
\[\frac{d}{dt}\Big|_{X}\Omega(x,t)=\frac{\partial\Omega}{\partial t}+(\nabla \Omega)^T\frac{dx}{dt}\Big|_{X}=\frac{\partial\Omega}{\partial t}+(\nabla\Omega)^Tu, \]
which is precisely is the material derivative. Putting all of this together and collecting coefficients in \eqref{lietransport1} gives exactly the higher-dimensional $b$-equation \eqref{b-equation}. From this interpretation, one can see the sources of the convection, stretching, and expansion terms. The convection term arises from the material derivative of the coefficient function $\Omega$, the stretching term arises from differentiating the line element, while the expansion term arises from differentiating the volume element.

We now present a proposition that highlights the transport structure of the $b$-equation.

\begin{proposition}[Lie-transport]\label{prop:lietransport}
    For $b\in\mathbb{Z}$, the higher-dimensional $b$-equation \eqref{b-equation} can be written in the form 
    \begin{equation}
    \left(\frac{\partial}{\partial t}+\mathcal{L}_u\right)\left(\Omega\cdot dx\otimes dV^{\otimes (b-1)}\right)=0.
\end{equation}
\end{proposition}

\begin{proof}\label{lietransportvector}
Using Cartan's identity $\mathcal{L}_U\alpha=\iota_U(d\alpha)+d(\iota_U\alpha)$, one readily finds the Lie derivative of the momentum one-form $\Omega\cdot\,dx=\Omega_j\,dx^j$ to be
\[\mathcal{L}_U(\Omega\cdot\,dx)=\left(U^k\frac{\partial\Omega_j}{\partial x^k}+\Omega_k\frac{\partial U^k}{\partial x^j}\right)dx^j=\left(\nabla_U\Omega+(\nabla U)^T\Omega\right)\cdot\,dx,\]
with the usual implied summation over repeated indices. By the product rule and the definition of divergence, we have
\[\mathcal{L}_U\left(dV^{\otimes (b-1)}\right)=(b-1)\operatorname{div}(u)\,dV^{\otimes(b-1)}.\]
Putting all of this together yields
\[\left(\frac{\partial\Omega}{\partial t}+\nabla_U\Omega+(\nabla U)^T\Omega+(b-1)\operatorname{div}(u)\Omega\right)\cdot dx\otimes dV^{\otimes(b-1)}=0,\]
which is precisely the higher-dimensional $b$-equation.
\end{proof}

 The previous proposition shows that the higher-dimensional $b$-equation expresses the invariance of the Lie-transport of the momentum one-form density $\Omega\cdot dx\otimes dV^{\otimes (b-1)}.$ This transport equation also satisfies a Kelvin-Noether theorem, akin to Kelvin's circulation theorem for perfect fluids, which states that the circulation of the velocity field along a material moving contour $C(t)$ is conserved in time. An analogous result for the $b$-equation is given in the following theorem. 

 \begin{theorem}[Kelvin-Noether]\label{kelvinnoether}
     Let $U(x,t)$ be the velocity solution to the higher-dimensional $b$-equation with corresponding momentum  $\Omega(x,t)$. If $C(u)$ is a material loop in $(dV)^{\otimes b-1}$ that is advected by the flow, then 
     \[\frac{d}{dt}\oint_{C(u)}\frac{1}{\rho^{b-1}}(\Omega\cdot\,dx)=0,\]
     where $\rho(x,t)$ is the density of the fluid. 
 \end{theorem}

 \begin{proof}
     By Proposition \ref{prop:lietransport}, the $b$-equation is a momentum transport equation. The flow therefore conserves mass, but not volume, and we can define the mass density $\rho$ by $dV_0=\rho(x,t)\,dV$, where $dV_0$ is the volume element at time $t=0$. This implies that the mass is Lie-transported by the velocity $u$:
     \[\left(\frac{\partial}{\partial t}+\mathcal{L}_u\right)(\rho(x,t)\,dV)=0.\]
    From this in conjunction with Proposition \ref{prop:lietransport}, it follows that the $b$-equation is equivalent to
     \begin{align*}
    \left(\left(\frac{\partial}{\partial t}+\mathcal{L}_u\right)\frac{1}{\rho^{b-1}}(\Omega\cdot dx)\right)\otimes (\rho\,dV)^{\otimes (b-1)}=0.
    \end{align*}
     Hence if $C(u)$ is advected by the flow, it holds that
     \begin{align*}
         \frac{d}{dt}\oint_{C(u)}\frac{1}{\rho^{b-1}}(\Omega\cdot\,dx)=\oint_{C(u)}\left(\frac{\partial}{\partial t}+\mathcal{L}_u\right)\left(\frac{1}{\rho^{b-1}}(\Omega\cdot\,dx)\right)=0.
     \end{align*}
 \end{proof}

\section{Local well-posedness in \texorpdfstring{$H^s(\mathbb{R}^n,\mathbb{R}^n)$}{Hs} and \texorpdfstring{$H^\infty(\mathbb{R}^n,\mathbb{R}^n)$}{Hinf}}\label{wellposednesssection}

We will now use Lagrangian coordinates to rewrite the $b$-equation \eqref{eqn:bequationdimn} as an ODE on the tangent bundle $T\mathcal{D}^s(\mathbb{R}^n)$ for $s>n/2+1$ with $s\geq r$, where $r$ is the order of the Fourier multiplier $A$. By showing that the geodesic spray is smooth, i.e. the right side of the ODE is smooth, we can apply Picard-Lindel\"{o}f to obtain local well-posedness of the geodesic flow on $T\mathcal{D}^s(\mathbb{R}^n)$ since it is a Banach manifold. This immediately yields local well-posedness in the corresponding Sobolev space, namely $H^s(\mathbb{R}^n,\mathbb{R}^n)$ when $s<n/2+1.$ Remarkably, there is no loss of spatial regularity during the time evolution of the $b$-equation, which allows us to transfer this result to the smooth category, obtaining local well-posedness in $H^\infty(\mathbb{R}^n,\mathbb{R}^n)$.

\begin{proposition}\label{system1}
Let $J\subset\mathbb{R}$ be an open interval containing $0$. Let $\varphi$ be the Lagrangian flow of the $H^s$ time-dependent vector field $u:J\times\mathbb{R}^n\to\mathbb{R}^n$ and define $v:=u\circ\varphi$. Then $u$ is a solution to the $b$-equation \eqref{eqn:bequationdimn} with $A$ an invertible linear operator if and only if $(\varphi,v)$ is a solution of 
\begin{equation}\label{cauchyprob3}
\begin{cases}
\varphi_t=v\\
v_t=S_\varphi(v)
\end{cases}
\end{equation}
where
\[S_\varphi=R_\varphi\circ S\circ R_{\varphi^{-1}}\quad\text{and}\quad S(u)=A^{-1}\left([A,\nabla_u]u-(\nabla u)^T(Au)-(b-1)\operatorname{div}(u)Au\right)\]
\end{proposition}

\begin{proof}
If $\varphi$ is the Lagrangian flow of $u$, then $\varphi_t=u\circ\varphi=:v$. Hence, using equation \eqref{eqn:bequationdimn}, we have
\begin{align*}
v_t&=(u_t+\nabla_uu)\circ\varphi\\
&=A^{-1}\left(A(\nabla_uu)-\nabla_u(Au)-(\nabla u)^T(Au)-(b-1)\operatorname{div}(u)Au\right)\circ\varphi\\
&=A^{-1}\left([A,\nabla_u]u-(\nabla u)^T(Au)-(b-1)\operatorname{div}(u)Au\right)\circ\varphi\\
&=R_\varphi(A^{-1}\left([A,\nabla_u]u-(\nabla u)^T(Au)-(b-1)\operatorname{div}(u)Au\right).
\end{align*}
Since $u=v\circ\varphi^{-1}=R_{\varphi^{-1}}(v)$, the result follows.
\end{proof}

One can now show that the vector field $F:\mathcal{D}^s(\mathbb{R}^n)\times H^s(\mathbb{R}^n,\mathbb{R}^n)\to H^s(\mathbb{R}^n,\mathbb{R}^n)\times H^s(\mathbb{R}^n,\mathbb{R}^n)$ defined by $(\varphi,v)\mapsto (v,S_\varphi(v))$
is smooth in precisely the same way as Bauer, Escher, and Kolev did for the EPDiff equation in \cite{bauer2015local}. We emphasize that composition and inversion are not smooth on $\mathcal{D}^s(\mathbb{R}^n)$, as it is not a Lie group, so that smoothness of $F$ does not follow immediately from the smoothness of $S.$  Note that $F$ is essentially the coordinate expression for the geodesic spray. The tangent bundle is trivial $T\mathcal{D}^s(\mathbb{R}^n)\cong \mathcal{D}^s(\mathbb{R}^n)\times H^s(\mathbb{R}^n,\mathbb{R}^n)$, and hence $F$ is a second-order vector field on $T\mathcal{D}^s(\mathbb{R}^n)$. To demonstrate smoothness of the spray, we follow the presentation in \cite{bauer2015local}. 

\begin{proposition}\label{thm:smoothspray}
Let $A$ be a Fourier multiplier in the class $\mathcal{E}^r$ for $r\geq1$. Let $s>n/2+1$ with $s\geq r$. Then the geodesic spray $(\varphi,v)\mapsto S_\varphi(v)$ where
\[S_\varphi=R_\varphi\circ S\circ R_{\varphi^{-1}}\quad\text{and}\quad S(u)=A^{-1}\left([A,\nabla_u]u-(\nabla u)^T(Au)-(b-1)\operatorname{div}(u)Au\right)\]
extends smoothly from $T\operatorname{Diff}_{H^\infty}(\mathbb{R}^n)$ to $T\mathcal{D}^s(\mathbb{R}^n).$
\end{proposition}

\begin{proof}
Define $P(u):=[A,\nabla_u]u$, $Q(u):=(\nabla u)^T(Au)$, and $R(u):=\operatorname{div}(u)Au$ so that
\[S_{\varphi}(v)=A^{-1}_{\varphi}\left(P_{\varphi}(v)-Q_{\varphi}(v)-(b-1)R_{\varphi}(v)\right).\]
The proof therefore reduces to establishing that the four mappings
\[(\varphi,v)\mapsto P_{\varphi}(v),\qquad (\varphi,v)\mapsto Q_{\varphi}(v),\qquad (\varphi,v)\mapsto R_{\varphi}(v),\qquad (\varphi,w)\mapsto A^{-1}_{\varphi}(w)\]
are smooth. 
\begin{enumerate}
\item  By Lemma 4.10 of \cite{bauer2015local}, the first Frech\'{e}t differential of $A_\varphi$ is given by 
\[\partial_\varphi A_\varphi(v,v)=A_{1,\varphi}(v,v)=-P_\varphi(v),\]
and hence 
\[(\varphi,v)\mapsto P_\varphi(v),\qquad \mathcal{D}^s(\mathbb{R}^n)\times H^s(\mathbb{R}^n,\mathbb{R}^n)\to H^{s-r}(\mathbb{R}^n,\mathbb{R}^n)\]
is smooth because $A_\varphi:=R_\varphi\circ A\circ R_{\varphi^{-1}}$ is smooth for Fourier multipliers $A$ of class $\mathcal{E}^r.$
\item We have $Q_\varphi(v)=R_\varphi\left((\nabla(v\circ\varphi^{-1}))^T\right)A_\varphi(v).$ In a local chart, one finds
\[\left((\nabla(v\circ\varphi^{-1}))^T\circ\varphi\right)^i_k(x)=\delta^{ij}\delta_{kl}[(d\varphi(x))^{-1}]^m_j\partial_mv^l(x),\]
where $(d\varphi(x))^{-1}$ is the inverse of the Jacobian $d\varphi(x)$. It therefore consists of products of these matrix entries with the partial derivatives $\partial_p v^p$. Hence,
\[(\varphi,v)\mapsto (\nabla(v\circ\varphi^{-1}))^T\circ\varphi \]
\[\mathcal{D}^s(\mathbb{R}^n)\times H^s(\mathbb{R}^n,\mathbb{R}^n)\to \mathcal{L}(H^{s-1}(\mathbb{R}^n,\mathbb{R}^n),H^{s-1}(\mathbb{R}^n,\mathbb{R}^n))\]
is smooth since multiplication in $H^{s-1}(\mathbb{R}^n,\mathbb{R}^n)$ is smooth when $s>n/2+1.$ Now since pointwise multiplication extends to a bounded bilinear map
\[H^{s-1}(\mathbb{R}^n,\mathbb{R}^n)\times H^{s-r}(\mathbb{R}^n,\mathbb{R}^n)\to H^{s-r}(\mathbb{R}^n,\mathbb{R}^n)\]
for $s>n/2+1$ and $0\leq s-r\leq s-1$ (by Proposition \ref{sobolevproperties}), and $(\varphi,v)\mapsto A_\varphi(v)$ is smooth by hypothesis, we conclude that \[(\varphi,v)\mapsto Q_\varphi(v),\qquad \mathcal{D}^s(\mathbb{R}^n)\times H^s(\mathbb{R}^n,\mathbb{R}^n)\to H^{s-r}(\mathbb{R}^n,\mathbb{R}^n)\]
is smooth. 
\item We have $R_\varphi(v)=(\operatorname{div}(v\circ\varphi^{-1})\circ\varphi)A_\varphi(v).$ Just note that
\[\operatorname{div}(v\circ\varphi^{-1})\circ\varphi(x)=[(d\varphi)^{-1}]^j_i\partial_jv^i(x),\]
and so we may conclude just as in part 2. that
\[(\varphi,v)\mapsto\operatorname{div}(v\circ\varphi^{-1})\circ\varphi,\qquad \mathcal{D}^s(\mathbb{R}^n)\times H^s(\mathbb{R}^n,\mathbb{R}^n)\to H^{s-1}(\mathbb{R}^n,\mathbb{R})\]
is smooth and hence so too is
\[(\varphi,v)\mapsto R_\varphi(v),\qquad \mathcal{D}^s(\mathbb{R}^n)\times H^s(\mathbb{R}^n,\mathbb{R}^n)\to H^{s-r}(\mathbb{R}^n,\mathbb{R}^n).\]
\item The set of isomorphisms $\operatorname{Isom}(H^s(\mathbb{R}^n,\mathbb{R}^n),H^{s-r}(\mathbb{R}^n,\mathbb{R}^n))$ is open in \[\mathcal{L}(H^s(\mathbb{R}^n,\mathbb{R}^n),H^{s-r}(\mathbb{R}^n,\mathbb{R}^n)).\]

Moreover, the inversion map $P\mapsto P^{-1}$ from
\[\operatorname{Isom}(H^s(\mathbb{R}^n,\mathbb{R}^n),H^{s-r}(\mathbb{R}^n,\mathbb{R}^n))\to \mathcal{L}(H^{s-r}(\mathbb{R}^n,\mathbb{R}^n),H^s(\mathbb{R}^n,\mathbb{R}^n))\]
is smooth. Also, $A_\varphi\in\operatorname{Isom}(H^s(\mathbb{R}^n,\mathbb{R}^n),H^{s-r}(\mathbb{R}^n,\mathbb{R}^n))$ for every $\varphi\in\mathcal{D}^s(\mathbb{R}^n)$ and the map $\varphi\mapsto A_\varphi$ from
\[D^s(\mathbb{R}^n)\to \operatorname{Isom}(H^s(\mathbb{R}^n,\mathbb{R}^n),H^{s-r}(\mathbb{R}^n,\mathbb{R}^n)) \]
is smooth by hypothesis. Hence, $(\varphi,w)\mapsto A^{-1}_\varphi(w)$
from $H^{s-r}(\mathbb{R}^n,\mathbb{R}^n)\to H^s(\mathbb{R}^n,\mathbb{R}^n)$
is smooth.
\end{enumerate}
Since each of the individual maps is smooth, their composition in the form of $S_\varphi(v)$ is smooth.
\end{proof}

\begin{remark}
    The assumption on the order of the Fourier multiplier is sharp, as pointed out in \cite{bauer2015local}. The reason is simple: the term $\operatorname{div}(u)$ is of order one, and it has no hope of becoming smooth unless the operator $A^{-1}$ is a smoothing operator of order at least one.
\end{remark}

Due to the smoothness of the spray $(\varphi,v)\mapsto S_\varphi(v)$, we obtain local existence of geodesics on the Banach manifold $T\mathcal{D}^s(\mathbb{R}^n)$ via Picard-Lindel\"{o}f.

\begin{theorem}
Let $A$ be a Fourier multiplier in the class $\mathcal{E}^r$ for $r\geq 1$. Let $s>n/2+1$ with $s\geq r$. Define 
\[S_\varphi=R_\varphi\circ S\circ R_{\varphi^{-1}}\quad\text{and}\quad S(u)=A^{-1}\left([A,\nabla_u]u-(\nabla u)^T(Au)-(b-1)\operatorname{div}(u)Au\right)\]
as before. For each $(\varphi_0,v_0)\in T\mathcal{D}^s(\mathbb{R}^n)=\mathcal{D}^s(\mathbb{R}^n)\times H^s(\mathbb{R}^n,\mathbb{R}^n)$, there exists a unique non-extendable solution
\[(\varphi,v)\in C^\infty(J_s(\varphi_0,v_0),T\mathcal{D}^s(\mathbb{R}^n))\]
of the Cauchy problem consisting of system \eqref{cauchyprob3} with initial data $\varphi_0=\varphi(0), v_0=v(0),$ where $J_s(\varphi_0,v_0)\subset\mathbb{R}$ is an open interval containing 0.
\end{theorem}

\begin{proof}
By Proposition \ref{thm:smoothspray}, the geodesic spray $F_s:(\varphi,v)\mapsto (v,S_\varphi(v))$ is smooth on the Banach space $T\mathcal{D}^s(\mathbb{R}^n)$. So by the Picard-Lindel\"{o}f theorem, it follows that for each $(\varphi_0,v_0)\in T\mathcal{D}^s(\mathbb{R}^n)$, there exists a unique solution $(\varphi,v)\in C^\infty(J_s(\varphi_0,v_0),T\mathcal{D}^s(\mathbb{R}^n))$ to the ODE
\[\frac{d}{dt}(\varphi,v)=F_s(\varphi,v)\]
on some maximal interval of existence $J_s(\varphi_0,v_0)$ which is open and contains 0. This is equivalent to a solution of the Cauchy problem stated in the theorem. 
\end{proof}

As a corollary, we obtain local well-posedness of the $b$-equation in $H^s(\mathbb{R}^n,\mathbb{R}^n)$ for appropriate $s$.

\begin{corollary}
    Let $A$ be a Fourier multiplier in the class $\mathcal{E}^r$ for $r\geq 1$. Let $s>n/2+1$ with $s\geq r$. The $b$-equation \eqref{eqn:bequationdimn} has for any initial data $u_0\in H^s(\mathbb{R}^n,\mathbb{R}^n),$ a unique non-extendable smooth solution 
    $u\in C^0(J,H^s(\mathbb{R}^n,\mathbb{R}^n))\cap C^1(J,H^{s-1}(\mathbb{R}^n,\mathbb{R}^n)).$
    The maximal interval of existence $J$ is open and contains $0.$
\end{corollary}

To complete the picture by obtaining local well-posedness of the system \eqref{cauchyprob3} on the \textit{Frech\'{e}t} manifold $T\operatorname{Diff}_{H^\infty}(\mathbb{R}^n),$ and hence local well-posedness of the $b$-equation on $H^\infty(\mathbb{R}^n,\mathbb{R}^n)$, we apply a classical result of Ebin and Marsden \cite{ebin1970groups}. The remarkable feature of this result is that the maximal interval of existence does not depend on the regularity parameter $s$, ultimately due to the equivariance of the spray. 

\begin{remark}
    Observe that by definition of the spray $F_s(\varphi,v)=(v,S_\varphi(v)),$ it is \textit{equivariant} by the right action of $\mathcal{D}^s(\mathbb{R}^n)$ on the tangent bundle $T\mathcal{D}^s(\mathbb{R}^n)=\mathcal{D}^s(\mathbb{R}^n)\times H^s(\mathbb{R}^n,\mathbb{R}^n)$. Concretely, 
    \[F(\varphi\circ\psi,v\circ\psi)=F(\varphi,v)\circ\psi,\qquad\text{for all $(\varphi,v)\in\mathcal{D}^s(\mathbb{R}^n)\times H^s(\mathbb{R}^n,\mathbb{R}^n),\quad \psi\in\mathcal{D}^s(\mathbb{R}^n)$}.\]
    
\end{remark}

\begin{proposition}[No Loss, No Gain]\label{lem:nolossnogain}
Consider the Cauchy problem consisting of system \eqref{cauchyprob3} with initial data initial data $(\varphi_0,v_0)\in T\mathcal{D}^s(\mathbb{R}^n)$. Let $J_s(\varphi_0,v_0)$ be the maximal interval containing 0 on which a solution to this Cauchy problem exists. Given $(\varphi_0,v_0)\in T\mathcal{D}^{s+1}(\mathbb{R}^n),$ it holds that 
\[J_{s+1}(\varphi_0,v_0)
=J_s(\varphi_0,v_0)\]
for $s> n/2+1$ and $s\geq r$.
\end{proposition}

\begin{proof}
If $(\varphi_0,v_0)\in T\mathcal{D}^{s+1}(\mathbb{R}^n)\subset T\mathcal{D}^s(\mathbb{R}^n)$, then we can solve \eqref{cauchyprob3} in both $T\mathcal{D}^{s+1}(\mathbb{R}^n)$ and $T\mathcal{D}^s(\mathbb{R}^n)$. Since $J_s(\varphi_0,v_0)$ is the largest interval on which the solution exists while maintaining $H^s$ regularity, we must have $J_{s+1}(\varphi_0,v_0)\subset J_s(\varphi_0,v_0).$ Let $\Phi_s$ be the flow of the spray $F_s$, and let $u$ be a constant vector field on $\mathbb{R}^n$. The flow $\psi^u_h(x):=x+hu$ of $u$ is not in $\mathcal{D}^s(\mathbb{R}^n)$ since $H^s(\mathbb{R}^n,\mathbb{R}^n)$ does not contain constant vector fields. But compose the flow with $\varphi\in\mathcal{D}^s(\mathbb{R}^n)$, and we get $\varphi\circ\psi^u_h\in\mathcal{D}^s(\mathbb{R}^n)$. We can therefore define the action of $\psi^u_h$ on $\mathcal{D}^s(\mathbb{R}^n)$ by
\[(\psi^u_h\cdot \varphi)(x):=\varphi(x+hu),\qquad \varphi\in\mathcal{D}^s(\mathbb{R}^n),\quad h\in\mathbb{R},\quad x,u\in\mathbb{R}^n.\]
This action on the diffeomorphism group naturally induces an action on the tangent bundle $T\mathcal{D}^s(\mathbb{R}^n)$ via
\[(\psi^u_h\cdot (\varphi,v))(x):=(\varphi(x+hu),v(x+hu)),\qquad (\varphi,v)\in T\mathcal{D}^s(\mathbb{R}^n),\quad h\in\mathbb{R},\quad x,u\in\mathbb{R}^n.\]
Observe that for $(\varphi,v)\in T\mathcal{D}^{s+1}(\mathbb{R}^n)$, the map
$h\mapsto \psi^u_h\cdot (\varphi,v)$ from $\mathbb{R}$ to $T\mathcal{D}^s(\mathbb{R}^n)$ is $C^1$ with 
\[\frac{d}{dh}(\psi^u_h\cdot (\varphi,v)(x))=(\nabla\varphi(x+hu)\cdot u,\nabla_uv(x+hu)).\] 
Hence, if $(\varphi_0,v_0)\in T\mathcal{D}^{s+1}(\mathbb{R}^n)$, we have
\[
\frac{d}{dh}\Big|_{h=0}\Phi_s(t,\psi^u_h\cdot(\varphi_0,v_0)) =\partial_{(\varphi,v)}\Phi_h(t,(\varphi_0,v_0))\cdot (\nabla\varphi_0(x)\cdot u,\nabla_uv_0(x)).\]
On the other hand, by equivariance of the spray $F_s$, we see that $F_s$ is invariant under each right-translation $R_{\varphi}$ where $\varphi\in\mathcal{D}^s(\mathbb{R}^n).$ The same is true for its flow $\Phi_s$:
\[\Phi_s(t,\psi^u_h\cdot(\varphi_0,v_0))=\psi^u_h\cdot \Phi_s(t,(\varphi_0,v_0)),\quad\text{for all}\quad t\in J_s(\varphi_0,v_0),\,h\in\mathbb{R}. \]
Thus we get 
\[\partial_{(\varphi,v)}\Phi_s(t,(\varphi_0,v_0))\cdot (\nabla\varphi_0(x)\cdot u,\nabla_uv_0(x))=(\nabla\varphi(t)\cdot u,\nabla_uv(t)).\]
But also $\partial_{(\varphi,v)}\Phi_s(t,(\varphi_0,v_0))\cdot (\nabla\varphi_0(x)\cdot u,\nabla_uv_0(x))\in H^s(\mathbb{R}^n,\mathbb{R}^n)\times H^s(\mathbb{R}^n,\mathbb{R}^n)$ for all constant vector fields $u$, and therefore
\[(\varphi(t),v(t))\in T\mathcal{D}^{s+1}(\mathbb{R}^n)\quad\text{for all}\quad t\in J_s(\varphi_0,v_0).\]
So if we start in the smaller space: $(\varphi_0,v_0)\in T\mathcal{D}^{s+1}(\mathbb{R}^n)\subset T\mathcal{D}^s(\mathbb{R}^n)$, and solve \eqref{cauchyprob3} in $T\mathcal{D}^s(\mathbb{R}^n)$, we find that the solution effectively remains in the smaller, higher-regularity space for all $t\in J_s(\varphi_0,v_0).$ From this we conclude
$J_s(\varphi_0,v_0)= J_{s+1}(\varphi_0,v_0)$.
\end{proof}

Proposition \ref{lem:nolossnogain} states that there is no loss of spatial regularity during the time evolution of \eqref{cauchyprob3}. This implies the following well-posedness result.

\begin{theorem}
Let $A$ be a Fourier multiplier of class $\mathcal{E}^r$ for $r\geq1$. Define 
\[S_\varphi=R_\varphi\circ S\circ R_{\varphi^{-1}}\quad\text{and}\quad S(u)=A^{-1}\left([A,\nabla_u]u-(\nabla u)^T(Au)-(b-1)\operatorname{div}(u)Au\right)\]
as before.
For each $(\varphi_0,v_0)\in T\operatorname{Diff}_{H^\infty}(\mathbb{R}^n)$, there exists a unique non-extendable solution
\[(\varphi,v)\in C^\infty(J,T\operatorname{Diff}_{H^\infty}(\mathbb{R}^n))\]
of the Cauchy problem consisting of system \eqref{cauchyprob3} with initial data $(\varphi_0,v_0)$ on the maximal interval of existence $J$, which is open and contains 0.
\end{theorem}
\begin{proof}
Proposition \ref{lem:nolossnogain} asserts that there is no loss of spatial regularity during the time evolution of \eqref{cauchyprob3}. There is also no gain of regularity in the sense that if $(\varphi_0,v_0)\in T\mathcal{D}^s(\mathbb{R}^n)$ such that $(\varphi(t_1),v(t_1))\in T\mathcal{D}^{s+1}(\mathbb{R}^n)$ for some $t_1\in J_s(\varphi_0,v_0)$, then it follows that $(\varphi_0,v_0)\in T\mathcal{D}^{s+1}(\mathbb{R}^n).$ This holds by reversing the time direction and applying uniqueness. This means that a solution keeps the exact regularity it starts with. Hence, if the initial conditions are smooth (i.e. $(\varphi_0,v_0)\in T\operatorname{Diff}_{H^\infty}(\mathbb{R}^n)$), then the solution will remain smooth under the time evolution of \eqref{cauchyprob3}. This proves the theorem.
\end{proof}

As a corollary, we obtain local well-posedness of the $b$-equation in the smooth category.

\begin{corollary}\label{mainresult}
The n-dimensional $b$-equation \eqref{b-equation} with inertia operator a Fourier multiplier of class $\mathcal{E}^r$ with $r\geq1$ has for any initial data $u_0\in H^\infty(\mathbb{R}^n,\mathbb{R}^n)$, a unique non-extendable smooth solution $u\in C^\infty(J,H^\infty(\mathbb{R}^n,\mathbb{R}^n))$. 
\end{corollary}

\begin{remark}
    Since the Sobolev inertia operator $\Lambda^{2s}=\textbf{op}((1+|\xi|^2)^s)$ is a Fourier multiplier of order at least $1$ when $s\ge1/2$, this well-posedness result applies to $\Lambda^{2s}$. 
\end{remark}

\bibliographystyle{abbrv}
\bibliography{refs}

\end{document}